\documentclass[11pt]{amsart}
\usepackage{amscd}
\newtheorem{theorem}{Theorem}[section]
\newtheorem{lem}[theorem]{Lemma}
\newtheorem{prop}[theorem]{Proposition}
\newtheorem{cor}[theorem]{Corollary}

\theoremstyle{definition}
\newtheorem{definition}[theorem]{Definition}

\theoremstyle{remark}

\numberwithin{equation}{section}

\newcommand{\Hor}{{\mathcal{H}}}
\newcommand{\V}{{\mathcal{V}}}
\newcommand{\ra}{\rightarrow}

\newcommand{\C}{{\mathbb{C}}}

\newcommand{\HH}{{\mathbb{H}}}
\newcommand{\CP}{{\mathbb{CP}}}

\newcommand{\HP}{{\mathbb{HP}}}

\newcommand{\KK}{\mathbb{K}}
\newcommand{\KP}{\mathbb{KP}}
\newcommand{\Z}{\mathbb{Z}}

\newcommand{\lb}{\langle}
\newcommand{\rb}{\rangle}

\newcommand{\mg}{\mathfrak{g}}
\newcommand{\mh}{\mathfrak{h}}
\newcommand{\mk}{\mathfrak{k}}
\newcommand{\mm}{\mathfrak{m}}

\renewcommand{\mp}{\mathfrak{p}}

\newcommand{\R}{\mathbb{R}}

\newcommand{\GG}{G_k(\KK^N)}
\newcommand{\GR}{G_k(\R^N)}

\newcommand{\Gf}{F_k(\R^N)}

\begin{document}

\newcommand{\spacing}[1]{\renewcommand{\baselinestretch}{#1}\large\normalsize}
\spacing{1.14}

\title{On pull-backs of the universal connection}

\author {Kristopher Tapp}
\begin{abstract}
Narasihman and Ramanan proved in~\cite{NR} that an arbitrary connection in a vector bundle over a base space $B$ can be obtained as the pull-back (via a correctly chosen classifying map from $B$ into the appropriate Grassmannian) of the universal connection in the universal bundle over the Grassmannian.  The purpose of this paper is to relate geometric properties of the classifying map to geometric properties of the pulled-back connection.  More specifically, we describe conditions on the classifying map under which the pulled-back connection: (1) is fat (in the sphere bundle), (2) has a parallel curvature tensor, and (3) induces a connection metric with nonnegative sectional curvature on the vector bundle (or positive sectional curvature on the sphere bundle).
\end{abstract}

\maketitle

\section{introduction}
Let $B$ denote an $n$-dimensional compact Riemannian manifold, let $\KK\in\{\R,\C,\HH\}$, and let $k\in\Z^+$.  To classify the $\KK^k$-vector bundles over $B$, one considers the Grassmannian of $k$-dimensional $\KK$-linear subspaces in $\KK^N$, denoted $\GG$,
for sufficiently large $N$.  The \emph{universal vector bundle} over $\GG$ has total space $\{(\sigma,V)\mid\sigma\in\GG, V\in\sigma\}$ and projection map $(\sigma,V)\mapsto\sigma$.  It is well-known that every $\KK^k$-bundle over $B$ is the pull-back of this universal vector bundle via some \emph{classifying map} $\varphi:B\ra\GG$, and that the homotopy class of $\varphi$ determines the isomorphism class of the pulled-back bundle.

This universal vector bundle has a natural connection defined such that the covariant derivative of a section $t\mapsto(\sigma(t),V(t))$ equals the section $t\mapsto\left(\sigma(t),V'(t)^{\sigma(t)}\right)$, where the superscript denotes the orthogonal projection onto the subspace.  Narasimhan and Ramanan proved in~\cite{NR} (see also~\cite{Schlafly}) that for sufficiently large $N$, this connection is \emph{universal} in the sense that every connection in the pulled-back bundle over $B$ can be obtained as a pullback of this universal connection by correctly choosing $\varphi$ within the homotopy class representing the bundle.  This theorem has seen many abstract applications within mathematics and physics, but to the best of our knowledge, it has never been used to prove or disprove the existence of connections in vector bundles with specific desirable geometric properties.

Let $\varphi:B\ra\GG$ be an explict classifying map, let $\pi_E:E\ra B$ denote the pulled-back bundle, let $\nabla$ denote the pulled-back connection, and let $R^\nabla$ denote its curvature tensor.  Let $\pi_1:E^1\ra B$ denote the sphere-bundle over $B$ formed from all unit-length vectors in $E$.

The first property we will study is \emph{fatness}, which was defined by Weinstein in~\cite{Weinstein} and studied by many authors (see~\cite{Ziller} for a survey).  To define fatness, it is useful to select a metric on $B$ and a rotationally symmetric ``fiber metric'' on $\KK^k$.  Together with $\nabla$, these choices induce a unique \emph{connection metric} on $E$ (and thus also on $E^1$).  The connection $\nabla$ called \emph{fat} if the sectional curvatures of all $\pi_1$-vertizontal planes are positive.  Notice that we're calling $\nabla$ \emph{fat} when the induced connection in the sphere bundle is fat as defined by Weinstein.  This fatness condition turns out not to depend on the choice of base metric or fiber metric, so fatness is a property only of the connection.

In Section~\ref{S:calc}, we will describe the general condition on $\varphi$ under which $\nabla$ is fat.  In the special case where $k=1$ and either $\KK=\C$ (so that $\varphi:B\ra\CP^{N-1}$) or $\KK=\HH$ (so that $\varphi:B\ra\HP^{N-1}$), fatness turns out to be equivalent to $\varphi(B)$ having bounded Wirtinger angles, defined as follows:
\begin{definition}\label{DD} Let $\KK\in\{\C,\HH\}$ and let $\mathcal{J}$ denote...
\begin{itemize}
\item[...]$\{J\}$ if $\KK=\C$, where $J$ denotes the standard almost complex structure on $\CP^N$,
\item[...]$\{I,J,K\}$ if $\KK=\HH$, where $\{I,J,K\}$ denotes a local coordinate expressions for the standard almost quaternionic structure on $\HP^N$.
\end{itemize}
Let $S\subset\KK\mathbb{P}^N$ be an immersed submanifold, $p\in S$ and $X\in T_pS$.  The \emph{Wirtinger angle}, $\theta(X)\in[0,\pi/2]$, is the maximum angle that a vector in $\text{span}_\R\{\mathfrak{J}X\mid\mathfrak{J}\in\mathcal{J}\}$ makes with $T_pS$.  If $\theta(X)<\pi/2$ for all $X\in TS$, then $S$ is said to have \emph{bounded Wirtinger angles}.
\end{definition}

\begin{theorem} \label{P:fat}If $k=1$ and $\KK\in\{\C,\HH\}$, then $\nabla$ is fat if and only if $\varphi$ is an immersion with bounded Wirtinger angles.
\end{theorem}

There are several well-studied conditions that imply bounded Wirtinger angles. For example, an immersion is called \emph{K\"ahler} (when $\KK=\C$) or \emph{quaternionic} (when $\KK=\HH$) if $\theta$ is constant at zero.  More generally, a submanifold of $\CP^N$ or $\HP^N$ is called \emph{slant} if $\theta$ is constant, or \emph{semi-slant} if its tangent bundle decomposes into two sub-bundles on which $\theta$ is constant respectively at zero and at another value.  In Section~\ref{S:lit}, we will survey some of the literature related to these conditions and it's implications to the search for fat connections.

Going the other direction, there are several known obstructions to fat connections.  For example, a trivial bundle can not admit a fat connection~\cite{Weinstein}, which implies that:
\begin{cor} A homotopically trivial immersed submanifold of $\CP^N$ or $\HP^N$ could not have bounded Wirtinger angles.
\end{cor}
Among $\HH$-bundles over $S^4$, only the Hopf bundle admits a fat connection~\cite{DR}, which implies:
\begin{cor} For any $N>0$, there is at most one element of $\pi_4(\HP^N)$ that admits an immersed representative $\varphi:S^4\ra\HP^N$ with bounded Wirtinger angles.
\end{cor}

In fact, if $N$ is large enough to ensure universality, then there is \emph{exactly} one such element.  Similarly, among circle bundles over the flag $F^6=SU(3)/T^2$, there is an infinite family which admit fat connections (corresponding to the positively curved Aloff Wallach spaces) and there are two which do not (the trivial bundle and the $W_{1,0}$ bundle); see~\cite{Ziller}.  Thus:
\begin{cor} For sufficiently large $N$, there are infinitely many homotopy classes of maps $F^6\ra\CP^N$ which admit an immersed representatives with bounded Wirtinger angles, and two which do not.
\end{cor}
We are not aware of any work specifically addressing the possible homotopy classes of submanifolds of $\KP^N$ with bounded Wirtinger angles, but the above discussion suggests that this question is both natural and subtle.

Next, we will study conditions on $\varphi$ under which $\nabla$ is \emph{parallel} or \emph{radially symmetric}:
\begin{definition}  A connection $\nabla$ is called \emph{parallel} if the covariant derivative of its curvature tensor vanishes; that is, $D_ZR^\nabla(X,Y)W=0$ for all $p\in B$, all $X,Y,Z\in T_p B$ and all $W\in E_p=\pi_E^{-1}(p)$.  The connection is called \emph{radially symmetric} if this condition holds when $Z=X$; that is, if $D_XR^\nabla(X,Y)W=0$ for all $p\in B$, all $X,Y\in T_p B$ and all $W\in E_p$.
\end{definition}

Notice that these conditions depend on both the connection and on the choice of metric on the base space.  We therefore assume for the remainder of this section that $\varphi$ is an immersion and that $B$ has the pull-back metric.  This added assumption sacrifices universality, since there is no reason to expect that a given metric on $B$ and a given connection can be simultaneously achieved from a single classifying map $\varphi$.

Strake and Walschap proved that if $B$ has positive sectional curvature, then any radially symmetric connection in any vector bundle over $B$ will induce a connection metric with nonnegative sectional curvature on the vector bundle~\cite{SW}.  However, all known examples of radially symmetric connections are parallel, and parallel connections appear to be rare.  For example, Guijarro, Sadun and Walschap proved in~\cite{parallel} that an $\R^k$ vector bundle over a compact simply connected irreducible symmetric space with a parallel connection must be isomorphic to an \emph{associated bundle}.  Since associated bundles trivially admit submersion metrics of nonnegative curvature, this result seems to limit one's ability to obtain topologically new examples of nonnegatively curved vector bundles using parallel connections (at least over symmetric base spaces).  For $\R^2$-bundles, radially symmetric is equivalent to parallel~\cite{STT}, but for higher rank bundles, the gap between these hypotheses is not well understood.

Let $II$ denote the second fundamental form of $\varphi(B)$ and let $S$ denote the shape operator, so that for $p\in B$, $X,Y\in T_pB\cong\varphi_*(T_p B)$ and $\eta\perp \varphi_*(T_p B)$, we have $\lb S_\eta X,Y\rb = \lb II(X,Y),\eta\rb$.
\begin{theorem}\label{P:par}
If $k=1$ and $\KK\in\{\C,\HH\}$, then $\nabla$ is parallel if and only if
$$\lb S_{(\mathfrak{J}X)^\perp}Y - S_{(\mathfrak{J}Y)^\perp}X,Z\rb = 0$$
for all $p\in B$, $X,Y,Z\in T_p B\cong\varphi_*(T_p B)$ and all $\mathfrak{J}\in\text{span}\{\mathcal{J}\}$, where ``$\perp$'' denotes the component orthogonal to $\varphi_*(T_pB)$.  Furthermore, $\nabla$ is radially symmetric if and only if the above condition is true in the special case $Z=X$.
\end{theorem}

In particular, if $\varphi(B)$ is totally geodesic, then $\nabla$ is parallel.  But in this case, $\varphi(B)$ is a symmetric space and the bundle is an associated bundle, as we will explain in Section~\ref{S:tg}.

There are examples in which $\nabla$ is parallel even though $\varphi(B)$ is not totally geodesic.  Specifically, if $k=1$ and $\varphi(B)$ is K\"ahler/quaternionic, then $\nabla$ is fat (as explained above) and also parallel (because $(\mathfrak{J}X)^\perp = (\mathfrak{J}Y)^\perp =0$).  The case $\KK=\HH$ is less interesting here because quaternionic implies totally geodesic.  But in the case $\KK=\C$, there are many examples of K\"ahler submanifolds of $\CP^{N-1}$ which are not totally geodesic.

The final property of the connection $\nabla$ which we wish to interpret in terms of the geometry of the classifying map $\varphi$ is the following inequality:
\begin{gather}
\text{For all }p\in B,\text{ all } X,Y\in T_pB, \text{ and all } W,V\in E_p, \notag\\
\lb (D_XR^\nabla)(X,Y)W,V\rb^2 \leq k_B(X,Y)\cdot |R^\nabla(W,V)X|^2,\label{ineq}
\end{gather}
where $k_B$ denotes the unnormalized sectional curvature of $B$.  This inequality was proven in~\cite{SW} to be a necessary condition for $\nabla$ (together with the given metric on $B$) to induce a connection metric with nonnegative sectional curvature on $E$.  Further, if this inequality is \emph{strictly} satisfied (for all orthonormal choices of $X,Y,V,W$), then it was proven in~\cite{T} that $\nabla$ induces a connection metric of nonnegative curvature in $E$ and of positive curvature in $E^1$.

This inequality relates the two previously-discussed properties of a connection: its left side vanishes if and only if $\nabla$ is radially symmetric, while on its right side, $\nabla$ is fat if and only if the expression $|R^\nabla(W,V)X|^2$ is strictly positive for orthonormal $X,W,V$~\cite[Equation 11]{SW}.  Thus, a fat radially symmetric connection over a positively curved base space will satisfy the inequality strictly, and will therefore induce a connection metric of nonnegative curvature on $E$ and of positive curvature on $E^1$.

Our translation of Inequality~\ref{ineq} becomes particularly simple for $\C^1$ and $\HH^1$-bundles:
\begin{theorem}\label{USA}If $k=1$ and $\KK\in\{\C,\HH\}$, then Inequality~\ref{ineq} is satisfied if and only if the following inequality is satisfied for all $p\in B$, $X,Y\in T_p B\cong\varphi_*(T_p B)$ and all $\mathfrak{J}\in\text{span}\{\mathcal{J}\}$:
\begin{equation}\label{loki}\lb S_{(\mathfrak{J}X)^\perp}Y - S_{(\mathfrak{J}Y)^\perp}X,X\rb^2\leq k_B(X,Y)\cdot|\text{proj}(\mathfrak{J}X)|^2,
\end{equation}
where ``proj'' denotes the orthogonal projection onto $\varphi_*(T_pB)$.  Further, Inequality~\ref{ineq} is \emph{strictly} satisfied (for orthonormal $X,Y,W,V$) if and only if Inequality~\ref{loki} is \emph{strictly} satisfied (for orthonormal $X,Y$ and unit-length $\mathfrak{J}$).
\end{theorem}
Inequality~\ref{loki} is clearly satisfied if the following quantities are \emph{both} sufficiently close to zero for all $p\in B$:
\begin{itemize}
\item $|S(p)|=\max\{|S_\eta X|\mid X\in\varphi_*(T_pB), \eta\perp\varphi_*(T_pB), |X|=|\eta|=1\},$
\item $\theta(p)=\max\{\theta(X)\mid X\in \varphi_*(T_pB)\}.$
\end{itemize}
In fact, the closer one of these quantities is to zero, the further the other one can move away from zero while still satisfying the inequality:
\begin{cor}\label{bable}Assume that the metric on $\KP^{N-1}$ is normalized to have maximal sectional curvature $1$.  If $|S(p)|^2<\frac{1}{16\tan^2\theta(p) + 8}$ for all $p\in B$, then Inequality~\ref{loki} is strictly satisfied, so $\nabla$ induces a connection metric of nonnegative curvature on $E$ and of positive curvature on $E^1$.
\end{cor}
For example, if the immersion is K\"ahler/quaternionic ($\theta=0$), then the inequality becomes $|S(p)|^2< 1/8$, which is the bound that insures that $B$ has positive curvature.  But as $\theta(p)\ra\pi/2$, the required upper bound on $|S(p)|^2$ goes to $0$.

The author is pleased to thank Luis Guijarro and Wolfgang Ziller for helpful comments and suggestions on this work.
\section{Calculations for the pull-back connection}\label{S:calc}
In this section, we assume that $\KK=\R$ and that $\varphi:B\ra\GR$ is an \emph{isometric immersion}.  Our goal is to describe $R^\nabla$ and $DR^\nabla$ in terms of the classifying map $\varphi$.

We require some notation.  Let $H\subset K\subset G$ equal the triple $O(N-k)\subset O(k)\times O(N-k)\subset O(N)$.  Let $\mh\subset\mk\subset\mg$ denote their Lie algebras.  Let $g_0$ be the biinvariant metric on $G=O(N)$ defined as $g_0(X,Y) = \frac 12\text{trace}(A\cdot B^T)$ for all $A,B\in\mg$.  We will sometime write $\lb A,B\rb_0$ to mean $g_0(A,B)$ and write $|A|^2_0$ to mean $g_0(A,A)$.  We endow $\GR=G/K$ with the normal homogeneous metric induced by $g_0$.  Denote $\mm=\mk\ominus\mh$ and $\mp=\mg\ominus\mk$, where ``$\ominus$'' denotes the $g_0$-orthogonal complement, so $\mg=\mh\oplus\mm\oplus\mp$.  Let $h:G\ra G/K=\GR$ denote the projection, which is a Riemannian submersion with respect to the above-mentioned metrics.

Let $p\in B$, and choose $g\in G$ such that $h(g)=\varphi(p)$.  Let $\mathfrak{T}\subset\mp$ denote the subspace such that $h_*(dL_g(\mathfrak{T}))=\varphi_*(T_pB)$.  For any $Z\in T_pB$, let $\tilde{Z}\in\mathfrak{T}$ denote the unique vector such that $h_*(dL_g(\tilde{Z})) = \varphi_* Z$.  Let $\widetilde{II}:\mathfrak{T}\times\mathfrak{T}\ra\mp\ominus\mathfrak{T}$ denote the lift of the second fundamental form, $II$, of $\varphi(B)$. In other words, $h_*(dL_g(\widetilde{II}(\tilde{X},\tilde{Y}))) = II(\varphi_*X,\varphi_*Y)$ for all $X,Y\in T_p B$.
\begin{prop}\hspace{.5in}\label{all}
\begin{enumerate}
\item $\nabla$ is fat if and only if $\left[\tilde{X},\alpha\right]^{\mathfrak{T}}\neq 0,$ and
\item $\nabla$ is parallel if and only if
$\left\langle \left[\tilde{X},\widetilde{II}(\tilde{Z},\tilde{Y})\right]
-\left[\tilde{Y},\widetilde{II}(\tilde{Z},\tilde{X})\right], \alpha\right\rangle_0=0, and$
\item Inequality~\ref{ineq} is satisfied if and only if the following is satisfied:
\begin{equation}
\left\langle\left[\tilde{X},\widetilde{II}(\tilde{X},\tilde{Y})\right] -
  \left[\tilde{Y},\widetilde{II}(\tilde{X},\tilde{X})\right],\alpha\right\rangle^2
\leq k_B(X,Y)\cdot \left|\left[\tilde{X},\alpha\right]^{\mathfrak{T}}\right|_0^2, \label{INEQ}
\end{equation}
\end{enumerate}
for all $p\in B$, all $X,Y,Z\in T_p B$ and all nonzero decomposable $\alpha\in\mm\cong so(k)\cong\Lambda^2(\R^k)$.
\end{prop}

Notice that the validity of these conditions do not depend on the choice of $g\in h^{-1}(\varphi(p))$, even though some of the individual terms do.

As we will explain later in this section, $\alpha$ represents the plane spanned by $W$ and $V$, and Inequalities~\ref{ineq} and~\ref{INEQ} match each other term-for-term in the obvious manner, from which parts (1) and (2) of the proposition follow.  Thus, all that is really required to prove Proposition~\ref{all} is to describe $R^\nabla$ and $DR^\nabla$ in terms of $\varphi$.  The remainder of this section is devoted to this task.

It is already clear that we intend to work in the setting of principle bundles, rather than vector bundles.  The total space of the universal principle $O(k)$-bundle over $\GR$ is the collection of ``frames'', i.e., ordered sets of $k$ orthonormal vectors in $\R^N$:
$$\Gf = O(N)/O(N-k)$$
Let $\pi:\Gf\ra\GR$ denote the projection map, which sends each frame to its span.

This universal principal bundle, $O(k)\hookrightarrow\Gf\stackrel{\pi}{\ra}\GR$, can be re-described as the following homogenous bundle:
\begin{equation}\label{got} K/H\hookrightarrow G/H\stackrel{\pi}{\ra}G/K.
\end{equation}
Notice that $\pi$ becomes a Riemannian submersion when $G/H$ and $G/K$ are endowed with the normal homogeneous metrics induced by $g_0$.  Let $\V$ and $\Hor$ denote the vertical and horizontal distributions of $\pi$.  We will refer to $\Hor$ as the ``universal connection'' because:

\begin{lem} $\Hor$ is a principle connection in the principle bundle; in fact, it equals the universal connection constructed in~\cite{NR}.
\end{lem}
\begin{proof}  Notice that $K=H\times O(k)$, so $H$ is normal in $K$, and $K/H=O(k)$.  Since $H$ commutes with $O(k)$, the right-$O(k)$-action on $G/H$ is well-defined and isometric.  Thus, the base space, $G/K$, of $\pi$ can be identified with $(G/H)/O(k)$.  Under this identification, $\pi$ is simply the quotient map from $G/H$ to $(G/H)/O(k)$.  In summary, the right $O(k)$-action on $G/H$ preserves each $\pi$-fiber (and therefore preserves $\V$) and is isometric (so it also preserves $\Hor$).  Thus, $\Hor$ is invariant under this principal $O(k)$-action, which makes it a principle connection.

Notice that the left $G$-action on itself induces a transitive isometric $G$-action on $G/H$ which sends $\pi$-fibers to $\pi$-fibers and therefore preserves $\Hor$ and $\V$.  Thus, to verify that $\Hor$ is the same as the universal connection constructed in~\cite{NR}, which has this same homogeneity property, it suffices to check a single point, which is straightforward.
\end{proof}

Consider the chain of Riemannian submersions $$G\stackrel{f}{\ra}G/H=\Gf\stackrel{\pi}{\ra}G/K=\GR,$$
and denote $h=\pi\circ f$.  Let $\omega_\Hor$ denote the connection form of $\Hor$, which is an $o(k)$-valued $1$-form on $\Gf$.  Let $\Omega_\Hor$ denote its curvature form, which is an $o(k)$-valued $2$-form on $\Gf$.

Let $g\in G$ be arbitrary, and denote $q=f(g)\in \Gf$.  Recall that $L_g:G\ra G$ induces an isometry $\Gf\ra\Gf$ which preserves $\V$ and $\Hor$.  Therefore,
$$\Hor_q = \{f_*(dL_gX)\mid X\in\mp\}\,\,\,\text{ and }\,\,\,\
  \V_q = \{f_*(dL_gV)\mid V\in\mm\}.$$
Notice that for any $V\in\mm\cong o(k)$, we have
\begin{equation}\label{ironman}\omega_{\Hor}(f_*(dL_gV)) = V,\end{equation}
simply because the action fields for the principle right $O(k)$-action on $\Gf$ are projections via $f_*$ of left-invariant fields on $G$.

\begin{lem}\label{L:OmegaH} For all $X,Y\in\mp$, we have: $\Omega_\Hor(f_*(dL_gX),f_*(dL_gY)) = \frac 12 [X,Y]^\mm$.
\end{lem}
\begin{proof}
We first consider the special case $g=e$ (the identity of $G$).  In this case, let $t\mapsto g(t)$ denote the geodesic in $G$ with $g(0)=e$ and $g'(0)=X$.  Let $t\mapsto Y(t)=dL_{g(t)}Y$ denote the left-invariant extension of $Y$ along this geodesic.  Notice that $t\mapsto (f\circ g)(t)$ is a $\pi$-horizontal geodesic in $\Gf$, and that $t\mapsto \hat{Y}(t) = f_*(Y(t))$ is a $\pi$-horizontal extension of $f_*Y$ along this horizontal geodesic.  Therefore the O'Neill tensor, $A$, of $\pi$ at the point $q$ satisfies:
$$A(f_*X,f_*Y) = \left(\hat{Y}'(0)\right)^\V = \left(f_*(Y'(0))\right)^\V = \frac 12 \left(f_*[X,Y]\right)^\V= \frac 12 f_*\left([X,Y]^\mm\right).$$

Now we return to the case where $g\in G$ is arbitrary.  The previous formula implies:
$$A(f_*(dL_gX),f_*(dL_gY)) = \frac 12 f_*\left(dL_g\left([X,Y]^\mm\right)\right)\in\V_{f(g)}.$$
Together with Equation~\ref{ironman}, this gives:
\begin{equation}\label{Omega_q}\Omega_\Hor(f_*(dL_g X),f_*(dL_g Y))
                   = \omega_{\Hor}(A(f_*(dL_gX),f_*(dL_gY)))
                   =\frac 12 [X,Y]^\mm.\end{equation}
\end{proof}

Let $\pi_P:P\ra B$ be the principle $O(k)$-bundle obtained as the pull-back via $\varphi$ of the universal frame bundle.  Explicitly, $P=\{(p,a)\in B\times \Gf\mid \varphi(p)=\pi(a)\}$ with projection $\pi_P(p,a) = p$.  Let $\overline{\varphi}:P\ra\Gf$ denote the corresponding bundle homomorphism, defined as $\overline{\varphi}(p,a)=a$.

Let $\omega$ denote the connection form of the principle connection in $P$ obtained as the pull-back via $\varphi$ of the universal connection in the universal frame bundle.  Let $\Omega$ denote the curvature form of $\omega$.  Thus, $\omega$ is an $o(k)$-valued 1-form on $P$, and $\Omega$ is an $o(k)$ valued $2$-form on $P$.

Let $\pi_E:E\ra B$ denote the vector bundle associated to $P$.  Explicitly, $E=P\times_{O(k)}\R^k$.  Let $p\in B$.  For any $x\in\pi_P^{-1}(p)$ and $U\in\R^k$, we let $x\diamond U\in E_p=\pi_E^{-1}(p)$ denote the image of $(x,U)$ under the projection $P\times\R^k\ra E$.  Let $\nabla$ the connection in $E$ associated to $\omega$, and let $R^\nabla$ denote its curvature tensor.

Let $X,Y\in T_p B$ be orthonormal, and let $W,V\in E_p$ be orthonormal.
For simplicity, we initially assume that $\varphi(p) = h(e)\in \GR$.  Let $a=f(e)\in \Gf$, and let $a'=(a,p)\in P$.  Notice that $\R^k$ can be identified with the fiber $E_p$ via the map which sends $\hat{U}\in\R^k$ to $U=a'\diamond\hat{U}\in E_p$. Let $\hat{W},\hat{V}\in\R^k$ be associated with $W,V\in E_p$ in this way.  There exists a unique $\alpha\in o(k)$ such that $\alpha\cdot \hat{W}=\hat{V}$, $\alpha\cdot\hat{V}=-\hat{W}$, and $\alpha\cdot\hat{U}=0$ for all $\hat{U}\in\R^k$ orthogonal to $\text{span}\{\hat{W},\hat{V}\}$, where the dots denote matrix multiplication.  We claim that:
\begin{equation}\label{notnot}\lb \beta\cdot\hat{W},\hat{V}\rb =  \lb\beta,\alpha\rb_0
\,\,\,\,\,\,\,\,\,\text{ for all }\beta\in o(k).\end{equation}
To see this, just choose an ordered orthonormal basis of $\R^k$ beginning with $\hat{W},\hat{V}$, and then express $\alpha,\beta$ in terms of the corresponding standard basis for $o(k)$.

For any $Z\in T_p B$, we let $\overline{Z}\in T_{a'}P$ denote its $\pi_P$-horizontal lift, and we let $\tilde{Z}\in\mp$ denote the unique vector such that $h_*\tilde{Z} = \varphi_* Z$.  In the following calculation, the maximum is taken over all unit-length $Z\in T_pB$:
\begin{eqnarray*}
|R^\nabla(W,V)X| & = & \max\lb R^\nabla(W,V)X,Z\rb = \max\lb R^\nabla(X,Z)W,V\rb \label{T1}\\
   & = & \max\lb \Omega(\overline{X},\overline{Z})\cdot \hat{W},\hat{V}\rb
     =
   \text{max}\lb\Omega(\overline{X},\overline{Z}),\alpha\rb_0 \\
   & = &  \text{max}\lb\Omega_\Hor
(\overline{\varphi}_*\overline{X},\overline{\varphi}_*\overline{Z}),\alpha\rb_0
     =   \frac 12\text{max}\lb[\tilde{X},\tilde{Z}]^\mm,\alpha\rb_0 \\
   & = & \frac 12\text{max}\lb[\tilde{X},\tilde{Z}],\alpha\rb_0
     =   \frac 12\text{max}\lb[\tilde{X},\alpha],\tilde{Z}\rb_0= \frac 12|[\tilde{X},\alpha]^\mathfrak{T}|_0.
\end{eqnarray*}
The last equality uses that $\varphi$ is an isometric immersion, so maximizing over all unit-length $Z\in T_p B$ is the same as maximizing over all unit-length $\tilde{Z}\in\mathfrak{T}$.  Recall that
where $\mathfrak{T}\subset\mp$ is the subspace such that $h_*(\mathfrak{T})=\varphi_*(T_pB)$.  In summary:
\begin{equation}\label{m1}
|R^\nabla(W,V)X| = \frac 12|[\tilde{X},\alpha]^\mathfrak{T}|_0.
\end{equation}

It remains to express the expression $\lb(D_ZR^\nabla)(X,Y)W,V\rb$ in terms of the geometry of $\varphi$. For this, let $t\mapsto c(t)$ denote the geodesic in $B$ with $c(0)=p$ and $c'(0)=Z$.  Let $t\mapsto\overline{c}(t)$ denote its $\pi_P$-horizontal lift beginning at $ a' \in P$.  We can write $\overline{c}(t) = (c(t),\beta(t))$ where $t\mapsto\beta(t)$ is the $\pi$-horizontal lift of $t\mapsto \varphi(c(t))$ to $\Gf$ beginning at $\beta(0)=a$.  Let $t\mapsto g(t)$ be the $h$-horizontal lift to $G$ of $t\mapsto\varphi(c(t))$ beginning at $g(0)=e$.  Define $W(t)=\overline{c}(t)\diamond\hat{W}\in E_{c(t)}$ and $V(t)=\overline{c}(t)\diamond\hat{V}\in E_{c(t)}$. Notice that $V(t)$ and $W(t)$ are parallel because $\overline{c}(t)$ is horizontal.

Let $X(t)$ and $Y(t)$ denote the parallel extensions of $X,Y$ along $t\mapsto c(t)$.  Let $\overline{X}(t)$ and $\overline{Y}(t)$ denote the $\pi_P$-horizontal lifts of these fields along $t\mapsto\overline{c}(t)$.  Let $\tilde{X}(t)$ and $\tilde{Y}(t)$ denote the $h$-horizontal lifts along $t\mapsto g(t)$ of the fields $t\mapsto \varphi_*X(t)$ and $t\mapsto \varphi_*Y(t)$.  Notice that:
$$R^\nabla(X(t),Y(t))W(t) = \overline{c}(t)\diamond\left(\Omega(\overline{X}(t),\overline{Y}(t))\cdot\hat{W}\right)
 \in E_{c(t)}.$$
In the following calculation, we will use $\frac{D}{dt}$ to denote covariant differentiation with respect to $\nabla$, and $\frac{d}{dt}$ to denote the usual differentiation of a path of vectors in the Euclidean spaces $\R^k$ and $o(k)$, and prime $'$ to denote covariant differentiation with respect to the Levi Civita connection in $(G,g_0)$.  With this notation, we have:
\begin{eqnarray*}
\lb (D_Z R^\nabla)(X,Y)W,V\rb
   & = & \left\langle \frac{D}{dt}\Big|_{t=0}
         \overline{c}(t)\diamond\left(\Omega(\overline{X}(t),\overline{Y}(t))
         \cdot\hat{W}\right),V\right\rangle \\
   & = & \left\langle \overline{c}(0)\diamond\left(\frac{d}{dt}\Big|_{t=0}
        \Omega(\overline{X}(t),\overline{Y}(t))\cdot\hat{W}\right),V\right\rangle \\
   & = & \left\langle \left(\frac{d}{dt}\Big|_{t=0}
         \Omega(\overline{X}(t),\overline{Y}(t))\right)\cdot\hat{W},\hat{V}\right\rangle \\
   & = &  \left\langle \frac{d}{dt}\Big|_{t=0}
         \Omega(\overline{X}(t),\overline{Y}(t)),\alpha\right\rangle_0
     =   \left\langle \frac{d}{dt}\Big|_{t=0}
         \Omega_{\Hor}(\overline{\varphi}_*\overline{X}(t),
         \overline{\varphi}_*\overline{Y}(t)),\alpha\right\rangle_0 \\
   & = & \frac 12 \left\langle \frac{d}{dt}\Big|_{t=0}
         \left[dL_{g(t)}^{-1}\tilde{X}(t) ,dL_{g(t)}^{-1}\tilde{Y}(t)\right]^\mm,\alpha\right\rangle_0 \\
   & = & \frac 12 \left\langle \left[\frac{d}{dt}\Big|_{t=0}
         dL_{g(t)}^{-1}\tilde{X}(t),\tilde{Y}\right]^\mm +\left[
         \tilde{X},\frac{d}{dt}\Big|_{t=0}dL_{g(t)}^{-1}\tilde{Y}(t)\right]^\mm,\alpha\right\rangle_0.
\end{eqnarray*}
To interpret these terms, let $\{E_i\}$ denote an orthonormal basis of $\mp$, and let $\{E_i(t)\}$ denote their left-invariant extensions along $g(t)$; that is, $E_i(t)=dL_{g(t)}E_i$.  Then,
\begin{eqnarray*}
\frac{d}{dt}\Big|_{t=0}dL_{g(t)}^{-1}\tilde{Y}(t)
  & = & \frac{d}{dt}\Big|_{t=0}\sum\lb dL_{g(t)}^{-1}\tilde{Y}(t),E_i\rb_0 E_i
    =   \frac{d}{dt}\Big|_{t=0}\sum\lb\tilde{Y}(t),E_i(t)\rb_0 E_i \\
  & = &  \sum\lb\tilde{Y}'(0),E_i\rb_0 E_i + \sum\lb\tilde{Y},E_i'(0)\rb_0 E_i\\
  & = & \tilde{Y}'(0) + \frac 12 \sum\lb\tilde{Y},[\tilde{Z},E_i]\rb_0 E_i \\
  & = & \tilde{Y}'(0) - \frac 12 \sum\lb E_i,[\tilde{Z},\tilde{Y}]\rb_0 E_i\\
  & = & \tilde{Y}'(0) - \frac 12[\tilde{Z},\tilde{Y}].
\end{eqnarray*}

Since $t\mapsto Y(t)$ is a parallel vector field along the geodesic $t\mapsto c(t)$ in $B$, we have $\tilde{Y}'(0) = \widetilde{II}(\tilde{Z},\tilde{Y}) + A_h(\tilde{Z},\tilde{Y})$, where $A_h$ denotes the O'Neill tensor of $h$.  In summary,
\begin{eqnarray*}
\left[\tilde{X},\frac{d}{dt}\Big|_{t=0}dL_{g(t)}^{-1}\tilde{Y}(t)\right]^\mm
   & = & \left[\tilde{X},\tilde{Y}'(0)-\frac 12[\tilde{Z},\tilde{Y}]\right]^\mm\\
   & = & \left[\tilde{X},\widetilde{II}(\tilde{Z},\tilde{Y})+A_h(\tilde{Z},\tilde{Y}) - \frac 12[\tilde{Z},\tilde{Y}]\right]^\mm\\
   & = & \left[\tilde{X},\widetilde{II}(\tilde{Z},\tilde{Y})\right]^\mm.
\end{eqnarray*}
The last equality follows from the fact that $\tilde{X}\in\mp$ while $A_h(\tilde{Z},\tilde{Y})\in\mk$ and $[\tilde{Z},\tilde{Y}]\in\mk$ (because $\mk\subset\mg$ is a symmetric pair).  We similarly have that:
$$ \left[\tilde{Y},\frac{d}{dt}\Big|_{t=0}dL_{g(t)}^{-1}\tilde{X}(t)\right]^\mm
    = \left[\tilde{Y},\widetilde{II}(\tilde{Z},\tilde{X})\right]^\mm.$$
Therefore, since $\alpha\in\mm$, we have:
\begin{equation}\lb (D_Z R^\nabla)(X,Y)W,V\rb = \frac 12\left\langle \left[\tilde{X},\widetilde{II}(\tilde{Z},\tilde{Y})\right]
-\left[\tilde{Y},\widetilde{II}(\tilde{Z},\tilde{X})\right], \alpha\right\rangle_0.\label{parallelcurv}\end{equation}

\begin{proof}[Proof of Proposition~\ref{all}]
By~\cite[Equation 11]{SW}, $\nabla$ is fat if and only if $|R^\nabla(W,V)X|>0$ for all $p\in B$, nonzero $X\in T_pB$ and all \emph{orthonormal} $W,V\in E_p$  (this condition depends only on $\nabla$, even though a metric on $B$ must be chosen in order for the expression $R^\nabla(W,V)X$ and its norm to be defined).  Also notice that $\nabla$ is parallel if and only if $\lb D_ZR^\nabla(X,Y) W,V\rb=0 $ for all $p\in B$, $X,Y\in T_p B$ and \emph{orthonormal} $W,V\in E_p$.  Also notice that Inequality~\ref{ineq} is satisfied if and only if it is satisfied for all \emph{orthonormal} choices of $X,Y,W,V$.

Proposition~\ref{all} now follows from Equations~\ref{m1} and~\ref{parallelcurv}.  Recall that in these equations, $W,V$ were assumed to be orthonormal, and $\alpha\in\mm$ was selected to represent $\text{span}\{W,V\}$ in the sense of Equation~\ref{notnot}.  In fact, the nonzero \emph{decomposable} elements of $o(k)\cong\Lambda^2(\R^k)$ are exactly the elements which represent planes in this manner.
\end{proof}
\section{Totally geodesic classifying maps}\label{S:tg}
In this section, we assume that $\KK=\R$ and that the classifying map $\varphi:B\ra\GR$ is a totally geodesic isometric imbedding.  By Proposition~\ref{all}, this implies that $\nabla$ is parallel.  If additionally $B$ has positive sectional curvature, then $\nabla$ induces a connection metric of nonnegative curvature on $E$~\cite{SW}.

Since $\GR$ contains many totally geodesic submanifolds (which have not yet been fully classified) including many with positive curvature, this might appear to be a hopeful source for topologically new examples of vector bundles which admit metrics of nonnegative curvature.  But we will explain in this section why no new examples can be obtained in this way.

First notice that $B$ is a symmetric space because it is a totally geodesic submanifold of the symmetric space $\GR=G/K$ (here, as before, $H\subset K\subset G$ denotes the triple $O(N-k)\subset O(k)\times O(N-k)\subset O(N)$).  More precisely, $B=G'/K'$ where $G'\subset G$ and $K'=G'\cap K$.  Let $\rho':K'\ra O(k)$ denote the composition of the inclusion map into $K$ with the projection onto the first factor.  The following was observed by Rigas in~\cite{Rigas}.
\begin{prop}[Rigas]\label{rig}The bundle $\pi_E:E\ra B$ is isomorphic to the associated bundle $G'\times_{\rho'}\R^k\ra G'/K'$.
\end{prop}

\begin{proof} Let $\rho:K\ra O(k)$ denote the projection onto the first factor.
The universal principal bundle $O(k)\hookrightarrow\Gf\ra\GR$ was re-described in Equation~\ref{got} as the homogenous bundle $K/H\hookrightarrow G/H\ra G/K.$  This homogeneous bundle can again be re-described as the associated bundle $K/H\hookrightarrow G\times_{\rho}(K/H)\ra G/K.$  This implies that the universal vector bundle over $\GR$ can be re-described as $\R^k\hookrightarrow G\times_\rho\R^k\ra G/K$.  Consider the following commutative diagram, in which the right arrows denote the natural inclusion maps:
$$\begin{CD} G'\times_{\rho'}\R^k @> >> G\times_{\rho}\R^k \\
           @V VV @VV V \\
           G'/K' @>\varphi>> G/K
\end{CD}$$
the bundle on the right is the universal vector bundle over $\GR=G/K$, so the bundle on the left is isomorphic to the pull-back bundle $\pi_E:E\ra B$, as desired.
\end{proof}

In conclusion, if the classifying map $\varphi:B\ra\GR$ is a totally geodesic isometric imbedding, then $\nabla$ is parallel, but in this case the bundle is isomorphic to an associated bundle, which trivially admits a submersion metric of nonnegative sectional curvature.

\section{The cases of $\C^1$ and $\HH^1$ bundles}\label{S:k1}
The primary technical difficulty in applying Proposition~\ref{all} is that ``bracketing with $\alpha$'' is difficult to interpret geometrically in general.  However, in this section, we will provide a very natural geometric interpretation in the special case where $k=1$ and $\KK\in\{\C,\HH\}$ (so that $\varphi:B\ra\CP^{N-1}$ or $\varphi:B\ra\HP^{N-1}$).  Theorems~\ref{P:fat},~\ref{P:par} and~\ref{USA} will follow from this interpretation.

Even though we assumed in the Section~\ref{S:calc} that $\KK=\R$, almost all of the calculations generalize in the obvious way to the cases $\KK\in\{\C,\HH\}$.  The only exception involves the manner in which $\alpha$ was chosen to represent a particular plane in the fiber.  Recall that we were given arbitrary orthonormal vectors $W,V\in\R^k$, and we were able to choose $\alpha\in o(k)$ to represent the plane $\text{span}\{W,V\}$ in the sense that:
$$\lb \beta\cdot W,V\rb =  \lb\beta,\alpha\rb_0
\,\,\,\,\,\,\,\,\,\text{ for all }\beta\in o(k).$$

To generalize this proof to the case $\KK=\C$ (respectively $\KK=\HH$), we would be given arbitrary $\R$-orthonormal vectors $W,V\in\KK^k$, and we would need to choose $\alpha\in u(k)$ (respectively $\alpha\in sp(k)$) so $\lb \beta\cdot W,V\rb_\R =  \lb\beta,\alpha\rb_0$ for all $\beta\in u(k)$ (respectively $\beta\in sp(k)$).  Unfortunately, this is not generally possible unless we additionally assume that $W\perp\text{span}\{\mathfrak{J}V\mid \mathfrak{J}\in\mathcal{J}\}$.

However, when $k=1$, there is no trouble with choosing $\alpha$ as desired.  In fact, the choice $\alpha=V\cdot\overline{W}$ works, and all of the calculations in the previous section go through.

More specifically, when $k=1$ and $\KK=\C$ (so that $\varphi:B\ra\CP^{N-1}$), the chain $H\subset K\subset G$ from Section~\ref{S:calc} becomes $SU(N-1)\subset S(U(1)\times U(N-1))\subset SU(N)$, and $\mm=u(1)$ is spanned by a unique (up to sign) unit-length vector $\alpha\in\mm$.  It is not hard to see that for any $q\in G/K=\CP^{N-1}$, the map $\text{ad}_\alpha:\mp\ra\mp$ (which sends $\tilde{X}$ to $[\alpha,\tilde{X}]$) induces an involution of $T_q(\CP^{N-1})$ that is well-defined in the sense that it is independent of the choice of $g\in h^{-1}(q)$ through which the lift $\tilde{X}$ is defined as in Proposition~\ref{all}.  In fact, this is one natural way in which to define the standard almost complex structure on $\CP^{N-1}$.  Therefore in Proposition~\ref{all}, bracketing with $\alpha$ can be interpreted as applying the almost-complex structure $J$.

When $k=1$ and $\KK=\HH$ (so that $\varphi:B\ra\HP^{N-1}$), the chain $H\subset K\subset G$ from the previous section becomes $Sp(N-1)\subset Sp(1)\times Sp(N-1)\subset Sp(N)$.  Identify $\mathcal{J}=\{I,J,K\}$ with an oriented orthonormal basis of $\mm=sp(1)=\text{Im}(\HH)$. For any $q\in G/K=\HP^{N-1}$, the triple of maps $\text{ad}_I,\text{ad}_J,\text{ad}_K:\mp\ra\mp$ induces a triple of involutions of $T_q(\HP^{N-1})$ which satisfy the familiar properties of an almost quaternionic structure.  Changing to a different $g\in h^{-1}(q)$ has the effect of conjugating to a different oriented orthonormal basis of $\mm$, so the family of triples:
$$\{\{\text{ad}_I,\text{ad}_J,\text{ad}_K\}\mid\{I,J,K\} \text{ is an oriented orthonormal basis of }\mm \}$$
determines a well-defined family of triples of involutions of $T_q\HP^{N-1}$.  This is one way to define the natural almost quaternionic structure on $\HP^{N-1}$.  Recall that on an almost quaternionic manifold, a choice of basis $\{I,J,K\}$ for the almost quaternionic structure generally only exists locally, which is reflected in the dependence on $g\in h^{-1}(q)$ described above.  In any case, bracketing with all possible $\alpha\in\mm$ can be interpreted in Proposition~\ref{all} as applying all possible elements of $\text{span}\{\mathcal{J}\}$.

\begin{proof}[Proof of Theorem~\ref{P:fat}] Recall that $\nabla$ is fat if and only if $|R^\nabla(W,V)X|>0$ for all $p\in B$, nonzero $X\in T_pB$ and all orthonormal $W,V\in E_p$.
If $\varphi$ is not an immersion, then there exits $p\in B$ and $X\in T_pB$ such that $\varphi_*X=0$, which implies that $R^\nabla(W,V)X=0$ for any choice of $W,V$.  Thus, $\nabla$ is not fat.

Next assume that $\varphi$ is an immersion, and choose the pull-back metric for $B$.  Let $\alpha\in\mm$ represent the plane $\text{span}\{W,V\}$ in the sense of Equation~\ref{notnot}.  Let $\mathfrak{J}\in\text{span}\{\mathcal{J}\}$ represent $\alpha$ as described previously in this section.  By Equation~\ref{m1},
$$2\left|R^\nabla(W,V)X\right| = \left|[\tilde{X},\alpha]^\mathfrak{T}\right|_0 = \left|(\mathfrak{J}X)^{\varphi_*(T_pB)}\right|\geq|X|\cdot\cos(\theta(X)),$$
so $\nabla$ is fat if and only if $\theta(X)<\pi/2$ for all nonzero $X\in TM$.
\end{proof}

\begin{proof}[Proof of Theorem~\ref{P:par}]
The connection $\nabla$ is parallel if and only if the following equals zero for all $p\in B$, $X,Y,Z\in T_pB$ and $W,V\in E_p$:
\begin{eqnarray*}2 \lb (D_Z R^\nabla)(X,Y)W,V\rb
   & = & \left\langle \left[\tilde{X},\widetilde{II}(\tilde{Z},\tilde{Y})\right]
        -\left[\tilde{Y},\widetilde{II}(\tilde{Z},\tilde{X})\right], \alpha\right\rangle_0\\
   & = & \left\langle \left[\alpha,\tilde{X}\right],\widetilde{II}(\tilde{Z},\tilde{Y})\right\rangle_0
        -\left\langle \left[\alpha,\tilde{Y}\right],\widetilde{II}(\tilde{Z},\tilde{X})\right\rangle_0\\
   & = & \left\langle \mathfrak{J}X,II(Z,Y)\right\rangle
        -\left\langle \mathfrak{J}Y,II(Z,X)\right\rangle\\
   & = & \left\langle S_{(\mathfrak{J}X)^\perp}Y,Z\right\rangle
         - \left\langle S_{(\mathfrak{J}Y)^\perp}X,Z\right\rangle\\
   & = & \left\langle S_{(\mathfrak{J}X)^\perp}Y - S_{(\mathfrak{J}Y)^\perp}X,Z\right\rangle,
\end{eqnarray*}
where $\alpha\in\mm$ represent the plane $\text{span}\{W,V\}$ in the sense of Equation~\ref{notnot}, and $\mathfrak{J}\in\text{span}\{\mathcal{J}\}$ represents $\alpha$ as described previously in this section. Furthermore, $\nabla$ is radially symmetric if and only if the above is true in the special case $Z=X$.
\end{proof}

Theorem~\ref{USA} follows immediately from the calculations of the previous two proofs.  It remains only to prove Corollary~\ref{bable}.

\begin{proof}[Proof of Corollary~\ref{bable}]
The inequality of Theorem~\ref{USA} is:
$$\lb S_{(\mathfrak{J}X)^\perp}Y - S_{(\mathfrak{J}Y)^\perp}X,X\rb^2\leq k_B(X,Y)\cdot|\text{proj}(\mathfrak{J}X)|^2.$$

The terms of this inequality at a point $p\in B$ are bounded as follows:
\begin{eqnarray*}
\lb S_{(\mathfrak{J}X)^\perp}Y - S_{(\mathfrak{J}Y)^\perp}X,X\rb^2
   & \leq & 4|S(p)|^2\sin^2\theta(p).\\
k_B(X,Y) & = & k_{\KP^N}(X,Y) +\lb II(X,X),II(Y,Y)\rb - |II(X,Y)|^2\\
    & \geq & k_{\KP^N}(X,Y) - 2|S(p)|^2\\
    & \geq & 1/4 - 2|S(p)|^2.\\
|\text{proj}(\mathfrak{J}X)|^2 & \geq & \cos^2\theta(p).
\end{eqnarray*}
Thus, the inequality is satisfied if:
$$4|S(p)|^2\sin^2\theta(p) \leq (1/4 - 2|S(p)|^2)\cos^2\theta(p),$$
which can be re-expressed as:
$$|S(p)|^2\leq\frac{1}{16\tan^2\theta(p) + 8}.$$
\end{proof}

\section{Well studied classes of immersions into $\CP^N$ and $\HP^N$}\label{S:lit}
Theorems~\ref{P:fat},~\ref{P:par} and~\ref{USA} empower one to construct connections with certain natural geometric properties in $\KK^1$-vector bundles over $B$ (with $\KK\in\{\C,\HH\}$) by constructing immersions of $B$ into $\KP^{N-1}$ that satisfy certain hypotheses.  There is a large body of literature on immersed submanifolds of projective spaces with natural properties.  Some of these properties overlap the hypotheses required by our theorems.  In this section, we review some of the literature and and discuss it's relevance to the search for nice connections in $\KK^1$ bundles.

As mentioned previously, if $\varphi$ is an isometric K\"ahler/quaternionic immersion, then $\nabla$ is fat and parallel.  If additionally $B$ has positive curvature, then the inequality is strictly satisfied, so $\nabla$ induces a connection metric of nonnegative curvature in $E$ and of positive curvature in $E^1$.  Since these are strong conclusions, it is worthwhile to begin by surveying some of the relevant literature on isometric K\"ahler/quaternionic immersions.  We start with the case $\KK=\C$.

Calabi's rigidity theorem from~\cite{Calabi} says that if $f_1,f_2:B\ra\CP^N$ are both isometric K\"{a}hler immersions, then there exists a unitary transformation $U$ of $\CP^N$ such that $f_2=U\circ f_1$.  Calabi further classified all isometric K\"{a}hler imbeddings of $\CP^n(c_1)$ into $\CP^N(c_2)$, where $c_1,c_2$ denote the constant holomorphic curvature.  For any fixed $n$, he proved there exists a countably infinite family of imbeddings, $f_i:\CP^n(c/i)\ra\CP^{N_i}(c)$, where $i\in\Z^+$ and $N_i=\frac{(n+i)!}{n!i!}-1$.  The map $f_i$ is sometimes called the $i^{\text{th}}$ Veronese imbedding. It is not totally geodesic if $i>1$.  Each $f_i$ induces a parallel fat connection in the pulled-back $\C^1$-bundle over $\CP^n$, and the total space, $E_i^1$, of the corresponding circle-bundle inherits a connection metric of positive sectional curvature.

These examples are not new.  The main result of~\cite{parallel} implies that any vector bundle over the symmetric space $\CP^n=SU(n+1)/S(U(1)\times U(n))$ with a parallel connection must be isomorphic to an associated bundle, which means it has the form $SU(n+1)\times_{\rho}\C\ra \CP^n$ for some representation $\rho:S(U(1)\times U(n))\ra U(1)$.  There is a one parameter family of such representations coming from powers of the determinant of $A\in U(n)\cong S(U(1)\times U(n))$: $\rho_j(A)=\det(A)^j$.  The total space of each circle bundle therefore has the following form for some $j\in\Z^+$:
\begin{equation}\label{franklin}E^1_j=SU(n+1)\times_{\rho_j}U(1),\end{equation}
which can be shown to be diffeomorphic to the lens space $S^{2n+1}/\Z_j$.  

Nakagawa and Takagi in~\cite{NT} classified the isometric K\"{a}hler imbeddings of all other compact simply connected irreducible Hermitian symmetric space into $\CP^N$.  For each such space, they obtained a countably infinite families of imbeddings analogous to the Veronese imbeddings.  Notice that $\CP^n$ is the only such space with positive curvature, so no new examples of connection metrics with positive curvature in circle bundles could be obtained by pulling back the universal connection via these imbeddings.  As above, these pulled-back connections are parallel and fat, and the bundles are associated bundles.

More recently, Di Scala, Ishi, and Loi studied isometric K\"ahler immersions of the form $f:B\ra\CP^N(1)$, where $B$ is a homogeneous K\"ahler manifold~\cite{DIL}.  They proved that $f$ must be an imbedding and that $B$ must be simply connected.  Moreover, they conjectured that (some rescaling of) any simply connected homogeneous K\"ahler manifold, $B$, whose associated K\"ahler form is integral must admit an isometric K\"ahler imbedding into $\CP^N(1)$ for some $N$.  This conjecture would imply that over each such space there exists a $\C^1$-bundle that admits a parallel fat connection.

There is no classification of the isometric K\"ahler immersions $f:B\ra\CP^N(1)$ for which $B$ has positive or nonnegative sectional curvature, except under added hypotheses.  For example, If $B$ complex dimension $\geq 2$ and sectional curvature $>1/8$, then Ros and Verstraelen proved that $f(B)$ must be totally geodesic~\cite{RV}.  If $B$ has positive sectional curvature and has holomorphic curvature $\geq 1/2$, then Ros proved that $f$ must be a one of a list of standard imbeddings~\cite{Ros}.  If $B$ has nonnegative sectional curvature and has complex codimension less than its complex dimension, then Shen obtained a similar conclusion~\cite{Shen}.  But without any added hypotheses, no classification is known.  Any new example would be interesting, especially if $B$ had positive sectional curvature, for then the pulled back circle bundle over $B$ would inherit positive sectional curvature as well.

There are other conditions on immersions (more general than the K\"ahler/quaternionic condition) that imply bounded Wirtinger angles. For example, an immersion $\varphi:B\ra\KP^N$ is called called \emph{slant} if $\theta$ is constant, or \emph{semi-slant} if its tangent bundle decomposes into two sub-bundles on which $\theta$ is constant respectively at zero and at another value.  Slant submanifolds were defined by Chen, who summarized the early results in his book~\cite{chen}.  The pullback of any proper slant (or semi-slant) immersion would yield a $\KK^1$ bundle with a fat connection.  The literature on slant and semi-slant submanifolds consists primarily of rigidity results.  However, Maeda, Ohnita and Udagawa in~\cite{MOU} constructed examples of slant submanifolds of $\CP^N$, including families of proper slant imbeddings of $\CP^n$ into $\CP^N$ which generalize the Veronese imbeddings.

\bibliographystyle{amsplain}

\end{document}